\documentclass[11pt, a4paper]{amsart}
\usepackage[a4paper]{geometry}
\usepackage[latin1,utf8]{inputenc}
\usepackage[spanish,english]{babel}
\usepackage[bitstream-charter]{mathdesign}
\usepackage[T1]{fontenc}
\usepackage{mathbbol}
\usepackage{enumitem}
\usepackage[alphabetic,initials]{amsrefs}
\usepackage{stmaryrd}
\usepackage[usenames,dvipsnames,svgnames,table]{xcolor}
\usepackage{mathtools}
\usepackage{graphicx}
\usepackage[poly,arrow,curve,matrix,color]{xy}
\usepackage{hyperref}
\usepackage{wrapfig}
\usepackage{comment}
\usepackage[toc]{appendix}
\usepackage{colonequals}
\usepackage{multicol}
\usepackage{pdfpages}
\usepackage{caption}
\excludecomment{dontcompile}

\theoremstyle{plain}
\newtheorem{theorem}{Theorem}[section]
\newtheorem*{theorem*}{Theorem}
\newtheorem{proposition}[theorem]{Proposition}
\newtheorem*{proposition*}{Proposition}
\newtheorem*{theoreme*}{Théorème}
\newtheorem*{frproposition*}{Proposition}

\newtheorem{lemma}[theorem]{Lemma}
\newtheorem*{lemma*}{Lema}
\newtheorem{corollary}[theorem]{Corollary}
\newtheorem*{lemme*}{Lemme}

\theoremstyle{remark}

\theoremstyle{definition}

\newtheorem*{definition*}{Definición}
\newtheorem*{frdefinition*}{Définition}

\newcommand\NN{\mathbb{N}}
\newcommand\ZZ{\mathbb{Z}}

\newcommand\KK{\mathbb{k}}
\newcommand\FF{\mathbb{F}}

\newcommand\orb[1]{\mathcal{O}_{#1}}
\newcommand\clorb[1]{\overline{\orb{#1}}}
\newcommand\func{\mathcal S}
\newcommand\gsch{\mathcal G}
\newcommand\defo[1]{\mathcal #1}
\newcommand\fser{\KK\llbracket t \rrbracket}
\newcommand\lser{\KK(\!(t)\!)}
\newcommand\corr[1]{\mathrm{T}(#1)}
\newcommand\extesc[2]{#1\otimes_{#2}\KK(\!(t)\!)}

\DeclareMathOperator\id{id}

\DeclareMathOperator\Tor{Tor}
\DeclareMathOperator\Hom{Hom}
\DeclareMathOperator\im{Im}

\DeclareMathOperator\GL{GL}
\DeclareMathOperator\gGL{GL_0}

\newcommand{\inc}[1][]{
\ifthenelse{\equal{#1}{}}{\tau}{\tau_{#1}}
}
\newcommand{\proj}[1][]{
\ifthenelse{\equal{#1}{}}{\pi}{\pi_{#1}}
}
\newcommand{\unit}[1][]{
\ifthenelse{\equal{#1}{}}{\eta}{\eta_{#1}}
}

\newcommand{\ass}[2][null]{
\ifthenelse{\equal{#1}{null}}{\mathsf{Ass}_{#2}}{\mathsf{Ass}_{#2}(#1)}
}

\newcommand{\alg}[2][null]{
\ifthenelse{\equal{#1}{null}}{\mathsf{Alg}_{#2}}{\mathsf{Alg}_{#2}(#1)}
}

\newcommand{\gass}[2][null]{
\ifthenelse{\equal{#1}{null}}{\mathsf{GrAss}_{#2}}{\mathsf{GrAss}_{#2}(#1)}
}
\newcommand{\galg}[2][null]{
\ifthenelse{\equal{#1}{null}}{\mathsf{GrAlg}_{#2}}{\mathsf{GrAlg}_{#2}(#1)}
}

\newcommand{\liea}[2][null]{
\ifthenelse{\equal{#1}{null}}{\mathsf{Lie}_{#2}}{\mathsf{Lie}_{#2}(#1)}
}

\title{On geometric degenerations and Gerstenhaber formal deformations}

\begin{document}

\begin{abstract}
   {
   We study the degeneration relations on the varieties of associative and Lie algebra 
structures on a fixed finite dimensional vector space and give a description of them in 
terms of Gerstenhaber formal deformations. We use this result to show how the orbit 
closure of the $3$-dimensional Lie algebras can be determined using homological algebra. 
For the case of finite dimensional associative algebras, we prove that the $N$-Koszul 
property is preserved under the degeneration relation for all $N\geq2$.
   }
\end{abstract}
\author{Sergio Chouhy}
\thanks{{\footnotesize Institute of Algebra and Number Theory, University of 
Stuttgart, Pfaffenwaldring 57, 70569 Stuttgart,
Germany. The author is a DAAD-ALEARG fellow. Email: \emph{schouhy@dm.uba.ar}}}
\maketitle

\textbf{2010 Mathematics Subject Classification: 13D10.} 

\textbf{Keywords:} deformations, degenerations, homological methods. 

\section{Introduction}
The deformation theory of algebraic structures has its origin in the 
sixties in analogy with the deformation theory of 
complex varieties developed by Kodaira, Nirenberg, Spencer and Kuranishi 
\cite{Ko05}.
Let $V$ be a finite dimensional vector space over an algebraically closed field $\KK$. 
The set $X$ of all possible structure coefficients of a given algebraic structure on $V$ 
is an algebraic variety. Nijenhuis, Richardson \cite{NR64} and Gabriel \cite{Ga74} 
established the basis for the geometric theory of algebraic structures, which consists in 
studying the geometry of $X$ and its relation with the algebraic properties of its 
points.

For instance, let $X$ be the set of all maps $\varphi\in \Hom_\KK(V\otimes_\KK V,V)$ that 
satisfy the associativity condition $\varphi\circ(\varphi\otimes\id) = 
\varphi\circ(\id\otimes\varphi)$. Thus, $X$ is the variety of associative 
algebra structures on $V$. The group $\GL(V)$ of linear isomorphisms acts on $X$ via the 
formula $g\cdot \varphi = g\circ \varphi\circ (g^{-1}\otimes g^{-1})$ and the orbits are 
in correspondence with the isomorphism classes of associative algebras of dimension 
$\dim_\KK(V)$. Subsets of $X$ consisting of algebras with particular properties usually 
have nice geometric attributes. For example, the properties of having a unit element or of 
being of finite representation type are open in $X$ for the Zariski topology \cite{Ga74}. 
There is also a strong relation between the algebraic properties of the points of $X$ and 
the local geometry of $X$. This can be seen for instance in \cite{GP95} where the 
authors proved, among other results, that the vanishing of the third Hochschild 
cohomology space $\mathsf{H}^3(A,A)$ of an algebra $A$ implies that it is a smooth point 
of $X$. In the same direction, if $\mathsf{H}^2(A,A)=0$, then its orbit is open 
\cite{Ga74}.

Given an algebra $A$ in $X$, we denote its $\GL(V)$-orbit by $\orb 
A$ and its Zariski closure by $\clorb A$. The set $\clorb A$ is $\GL(V)$-invariant, thus 
it is a union of orbits. An algebra $B$ contained in $\clorb A$ is called a 
\textit{degeneration of $A$}.
Given the close connection between the local geometry of $X$ and the 
algebraic properties of its points, algebras and their degenerations are strongly 
related. For example, given $i\geq 0$, the map $A\mapsto 
\dim_\KK \mathsf{H}^i(A,A)$ is upper semicontinuous \cite{GP95}. As a consequence, if $B$ 
is an algebra contained in $\clorb A$ such that $\mathsf{H}^i(B,B)$ vanishes, then so 
does $\mathsf{H}^i(A,A)$. Therefore the Hochschild cohomology dimension of $A$ -- as 
defined in \cite{Ha06} -- bounds that of its degenerations. See  
also \cite{Ge95}, where the author proved that the property of being of tame 
representation type is preserved under the degeneration relation and used it to determine 
the representation type of several algebras.

On the other hand, given an associative 
algebra structure $A$ on $V$, a \textit{formal deformation}, or \textit{a one-parameter 
family of deformations}, of $A$ is a $\fser$-algebra $\defo X$, free as 
$\fser$-module, whose specialization at $0$ is $A$. The formal deformation theory -- 
that is, the study of the formal deformations -- was initiated by Gerstenhaber in 
\cite{Ge64}. He proved that it is a natural setting where the Hochschild cohomology 
spaces have different interpretations. For instance, the obstructions for a 
formal deformation to exist lie on the third cohomology space, and the vanishing of the 
second cohomology space implies that its formal deformations are all isomorphic to the 
algebra $A\llbracket t \rrbracket$ of formal series with coefficients in $A$. In the 
latter case the algebra $A$ is called \textit{analytically rigid}.
Despite the clear geometric motivation for this theory, it has a formal point of view 
and there is no requirement for the formal deformations to be germs of analytic curves.

There is a connection between the formal and the geometric deformation theories. Namely, 
if an algebra $A$ is analytically rigid, then $\orb A$ is open. The converse holds if the 
characteristic of $\KK$ is zero and $\mathsf{H}^3(A,A)=0$. See \cite{GP95} for more 
details. It is the purpose of this paper to contribute to the understanding of the 
relation between these theories in a way we shall now explain.

In \cite{GO88}, Grunewald and O'Halloran proved a general result which for the case 
of associative algebras implies that $B$ is a degeneration of $A$ if and only if there 
exists an affine curve $C$ in $X$ passing through $B$ and generically contained in $\orb 
A$. Suppose that $\orb B\subseteq \clorb A$ and denote by $R$ the stalk of $C$ at $B$, by 
$\mathfrak m$ the maximal ideal of $R$ and by $L$ its fraction field. The two conditions 
that $C$ satisfies imply that there is an $R$-algebra $\defo X$, free of rank $n$, 
such that $\defo X\otimes_R L$ is isomorphic to $A\otimes_\KK L$ as $L$-algebras and such 
that $\defo X\otimes_R R/\mathfrak m$ is isomorphic to $B$ as $\KK$-algebras. Our 
starting point is the following observation: the completion of $R$ with respect to the 
$\mathfrak m$-adic topology is isomorphic to $\fser$, the algebra of formal series, and 
the completion $\widehat{\defo 
X}$ of $\defo X$ is thus a $\fser$-algebra, free of rank $n$, whose specialization at $0$ 
is $B$. That is, $\widehat{\defo X}$ is a Gerstenhaber formal deformation of $B$. 
Moreover, $\widehat{\defo X}\otimes_{\fser}\lser$ is isomorphic to $A\otimes_\KK\lser$ as 
$\lser$-algebras, where $\lser$ denotes the field of Laurent series. Our main result 
states that the 
existence of such a Gerstenhaber formal deformation is also a sufficient condition for 
the 
inclusion $\orb B \subseteq \clorb A$ to hold.

\begin{theorem*}
Let $X$ be the variety of associative algebra structures on a finite dimensional vector 
space, and let $A$ and $B$ be associative $\KK$-algebras regarded as points on $X$.
The orbit $\orb B$ is contained in $\clorb A$ if and only if there 
exists a Gerstenhaber formal deformation $\defo X$ of $B$ such that 
$\extesc{\defo 
X}{\fser}$ is isomorphic to $\extesc{A}{\KK}$ as $\KK(\!(t)\!)$-algebras. 
\end{theorem*}

Similar theorems hold for the varieties of finite dimensional graded associative 
algebras and finite dimensional Lie algebras. See Theorems \ref{teo:caso_as}, 
\ref{teo:caso_gras} and \ref{teo:casoLie}. These results are consequences of our Theorem 
\ref{teo:main}, which is a variant of the description of orbit 
closure given by Grunewald and O'Halloran.

From this point of view we study the family of $N$-Koszul algebras 
\cite{Be01}. We prove that the property of being $N$-Koszul is preserved under the 
degeneration relation, meaning that in case $\orb B \subseteq \clorb A$, if $B$ is 
$N$-Koszul then $A$ is $N$-Koszul. The case $N=2$ was already known and it is a corollary 
of a result of Drinfeld with important consequences in algebraic combinatorics 
\cite{SY97}. See Proposition \ref{prop:koszul} for more details. 
On the Lie algebra setting we show how to use our results to describe the sets of the 
form $\clorb{\mathfrak g}$ for $3$-dimensional Lie algebras using homological algebra 
tools. This description was already obtained by Agaoka in \cite{Aga99} by geometric means.

\subsection*{Acknowledgments} The author is deeply indebted to Andrea Solotar for her 
help in improving this article.

\section{General theorem}
First we recall the characterization of orbit closure of 
Grunewald and O'Halloran in the cases of our interest. As in \cite{GO88}, we use the term 
\textit{affine variety} for the -- non necessarily irreducible -- zero locus of an ideal 
of polynomials.

Let $\KK$ be an algebraically closed field, let $X$ be an affine 
variety over $\KK$, and let $G$ be an affine algebraic group acting on 
$X$. Define $\func$ and $\gsch$ to be the functors from the category of 
associative and commutative $\KK$-algebras to the category of sets given by 
$\func(A)=\Hom(\mathsf{Spec}(A),X)$ and $\gsch(A)=\Hom(\mathsf{Spec}(A),G)$ respectively.
Given $x\in\func(\KK)$, denote by $\orb x$ its orbit in $\func(\KK)$ under the action of 
$\gsch(\KK)$, and by $\clorb x$ its Zariski closure.

\begin{theorem}[\cite{GO88}]
\label{teo:GO}
 Let $x,y\in\func(\KK)$. The orbit $\orb y$ is contained in $\clorb x$ if and only if 
there exist a discrete valuation $\KK$-algebra $R$ and an element
$\defo X\in\func(R)$ such that the residue field of $R$ is $\KK$, its quotient 
field $L$ is finitely generated and of transcendence degree $1$ over $\KK$, and
\begin{align*}
 &\func(\inc[R])(\defo X) = g\cdot(\func(\inc[R]\circ\unit[R])(x)), \hskip0.5cm \text{ 
for some }g\in\gsch(L)\text{,}\\
 &\func(\proj[R])(\defo X) = y,
\end{align*}
where $\inc[R]:R\to L$ is the inclusion map, $\unit[R]:\KK\to R$ is the unit, and 
$\proj[R]:R\to\KK$ is the canonical projection.
\end{theorem}

Denote by $\fser$ the algebra of formal series and by $\lser$ its fraction field; the 
field of Laurent series. Our main theorem is the following variation of Theorem 
\ref{teo:GO}.

\begin{theorem}
\label{teo:main}
Let $x,y\in\func(\KK)$. The orbit $\orb y$ is contained in $\clorb x$ if and only if 
there exists an element $\defo X\in\func(\fser)$ such that the following equalities 
hold.
\begin{align}
\label{eqteo1}
&\func(\inc)(\defo X) = g\cdot(\func(\inc\circ\unit)(x)), \hskip0.5cm \text{ for some 
}g\in\gsch(\lser)\text{,}\\
\label{eqteo2}
 &\func(\proj)(\defo X) = y,
\end{align}
where $\inc:\fser\to\lser$ is the inclusion map, $\unit:\KK\to\fser$ is the unit map, and 
$\proj:\fser\to\KK$ is the canonical projection.
\end{theorem}

\begin{proof}
Suppose $\orb y\subseteq \clorb x$. Let $R$ and $\defo X$ be as in Theorem 
\ref{teo:GO}. Let $\mathfrak m$ be the maximal ideal of 
$R$. The completion of $R$ with respect to the $\mathfrak m$-adic topology 
is isomorphic to $\fser$ \cite{Ser79}*{Chapter II, Theorem 2}. Let $\varphi:R\to\fser$ be 
the inclusion map and let $\defo X'\colonequals \func(\varphi)(\defo X)$. A 
diagram chase argument shows that the element $\defo X'$ satisfies Equations 
\eqref{eqteo1} and \eqref{eqteo2}, thereby proving the \textit{only if} part of the 
statement.

Suppose there exist elements $\defo X$ in $\func(\fser)$ and $g$ in $\gsch(\lser)$ 
satisfying Equations \eqref{eqteo1} and \eqref{eqteo2}. Let $B_1$ and $B_2$ be the 
global sections of $X$ and $G$ respectively. For $i=1,2$ let $n_i\in\NN$ and 
$\pi_i:\KK[X_1,\dots,X_{n_i}]\to B_i$ be the 
projections. Define $c_j = \defo X\circ\pi_1(X_j)$, for $j=1,\dots,n_1$, and 
$g_k=g\circ\pi_2(X_k)$ and $g^k=(g^{-1})\circ\pi_2(X_k)$, for $k=1,\dots,n_2$.

Let $v:\lser^\times\to\ZZ$ be the valuation map. Define $\corr 
f=t^{-v(f)}f$ for $f\in\lser^\times$, and $\corr 0=0$. Note that $\corr f$ belongs to 
$\fser$ for all $f\in\lser$ and also $\proj(\corr f)\neq0$ for all $f\neq0$. Let $A$ be 
the 
subalgebra of $\fser$ generated by $t$, $c_j$, $\corr{g_k}$ and $\corr{g^k}$, with 
$j=1,\dots, n_1$ and $k=1,\dots,n_2$. Also, let $\proj[A]$ be the restriction of 
$\proj$ to $A$.
Since $c_j$ belongs to $A$ for all $j$, there exists $\defo{X}_A\in\func(A)$ such that 
$\func(\inc[A])(\defo{X}_A)=\defo X$, where $\inc[A]:A\to\fser$ is the inclusion map. As 
a consequence 
\begin{align}
\label{eq1}
&\func(\proj[A])(\defo{X}_A)=\func(\proj)(\defo X)=y.
\end{align}
Similarly, since all the elements and operations underlying the equality 
$\func(\inc)(\defo X) = g\cdot(\func(\inc\circ\unit)(x))$ belong to $A[t^{-1}]$, there 
exists $g_A\in\gsch(A[t^{-1}])$ such that 
\begin{align}
\label{eq2}
&\func(\inc[A]')(\defo X_A)=g_A\cdot(\func(\inc[A]'\circ\unit[A])(x)),
\end{align}
where $\inc[A]':A\to A[t^{-1}]$ is the inclusion map and $\unit[A]:\KK\to A$ is the unit 
map.

By the Noether normalization Lemma there exist elements
$\beta_1,\dots,\beta_r\in A$, algebraically independent over $\KK$, such that $A$ is 
integral over $\KK[\beta_1,\dots,\beta_r]$. Let $\mathfrak q$ be the kernel of 
$\proj[A]$. The ideal $\mathfrak 
q$ is maximal and so $\mathfrak q\cap\KK[ 
\beta_1,\dots,\beta_r]$ is maximal as well. Since $\KK$ is algebraically 
closed, we obtain that $\mathfrak 
q\cap\KK[\beta_1,\dots,\beta_r]=(\beta_1-\lambda_1,\dots,\beta_r-\lambda_r)$ 
for some elements $\lambda_1,\dots,\lambda_r$ in $\KK$. We may assume 
without loss of generality that $\mathfrak q\cap\KK[ 
\beta_1,\dots,\beta_r]=(\beta_1,\dots,\beta_r)$.

For every affine line $L\subset\KK^r$ with $0\in L$, denote by ${\mathfrak 
p}_L\subseteq(\beta_1,\dots,\beta_r)$ its prime ideal. By the 
going-down theorem there exists a prime ideal $\mathfrak q_L$ of $A$ such 
that ${\mathfrak q}_L\cap\KK[\beta_1,\dots,\beta_r]={\mathfrak p}_L$ and 
$\mathfrak q_L\subseteq \mathfrak q$. Using that $\KK^r$ is the union of all 
the lines passing through $0$, we obtain $\cap_{L}\mathfrak p_L=0$. Since $A$ 
is 
integral over $\KK[\beta_1,\dots,\beta_r]$, we deduce $\cap_L\mathfrak q_L=0$. 
As a consequence, there exists an affine line $L\subseteq \KK^r$ such that 
$t\notin\mathfrak q_L$ and $\mathfrak q_L\subseteq \mathfrak q$. Let $\mathfrak 
p\colonequals\mathfrak p_L$ and $\mathfrak q'\colonequals \mathfrak q_L$. The 
ring $\KK[\beta_1,\dots,\beta_r]/\mathfrak p$ is isomorphic to the polynomial 
ring in one variable $\KK[\beta]$.

Define $S=A/\mathfrak q'$. Observe that $S$ is integral over 
$\KK[\beta]$. The fact that $t$ does not belong to $\mathfrak q'$ implies that its class 
in $S$ does not vanish. The ideal $\mathfrak qS$ is maximal 
in $S$ and \[S/\mathfrak qS \cong A/\mathfrak q \cong \KK.\]
Let $\overline S$ be the 
integral closure of $S$ in its quotient field, which we denote by $L$, and let 
$\overline{\mathfrak q}$ be a maximal ideal in $\overline S$ 
such that $\overline{\mathfrak q} \cap S = \mathfrak q S$. The algebra $\overline S$ is 
an integrally closed domain of Krull dimension $1$. Moreover, by 
\cite{Ser00}*{Chapter III,Proposition 16}, it is a finitely 
generated $\KK$-algebra and so it is a Dedekind domain. Define $R=\overline 
S_{\overline{\mathfrak q}}$, the localization of $\overline S$ at $\overline{\mathfrak 
q}$. The quotient field of $R$ is also $L$, the quotient field of $S$, and therefore it 
is finitely generated and of transcendence degree $1$ over $\KK$. Let $\unit[R]$, 
$\proj[R]$ and $\inc[R]$ be respectively the unit map, the canonical projection onto the 
residue field and the inclusion map of $R$.

Since $\KK$ is algebraically closed and isomorphic to $S/\mathfrak qS$, we deduce 
that $\overline S/\overline{\mathfrak q}\overline S\cong S/\mathfrak qS$. It follows that
\[
 R/\overline{\mathfrak q}R
 \cong \overline{S}/\overline{\mathfrak q}\overline{S}
 \cong S/\mathfrak q S
 \cong A/\mathfrak q\cong \KK.
\] 
Consider the following commutative diagram
\[
 \xymatrix{
      & A/\mathfrak q \ar[dl]\ar[d]
      & A \ar[l]\ar[d]\ar[r] 
      & A[t^{-1}]\ar[d]\\
  \KK & S/\mathfrak qS \ar[l]\ar[d] 
      & S \ar[l]\ar[d]\ar[r] & S[t^{-1}]\ar[d]\\
      & R/\overline{\mathfrak q}R \ar[ul] & R \ar[l]\ar[r] & L
 }
\]
The maps $R\to\KK$ and $R\to L$ are respectively $\proj[R]$ and 
$\inc[R]$.
Let $\iota$ be the map from $A$ to $R$ and let $\iota'$ be the map from $A[t^{-1}]$ to 
$L$. Define
\begin{align*}
&\defo{X}'= \func(\iota)(\defo{X}_A), &\text{ and } & &g'=\gsch(\iota')(g_A).
\end{align*}
The commutativity of the above diagram and Equations \eqref{eq1} and 
\eqref{eq2} imply that 
\begin{align*}
&\func(\inc[R])(\defo{X}') = g'\cdot(\func(\inc[R]\circ\unit[R])(x)), & & 
\func(\proj[R])(\defo{X}') = y. 
\end{align*}	
From Theorem \ref{teo:GO} we deduce that $\orb y\subseteq \clorb x$.
\end{proof}

\section{Gerstenhaber formal deformations}
\label{sec:gerst}
Here we briefly recall the notion of Gerstenhaber formal deformations of associative 
and Lie algebras and some fundamental results that we will need in the
forthcoming sections. For a complete account on the subject we refer to the original 
paper \cite{Ge64}.

For a $\KK$-vector space $V$, denote by $V\llbracket t\rrbracket$ 
the space of formal series with coefficients in $V$.

\subsection*{Associative algebras} Let $A$ be an associative algebra over $\KK$ with 
underlying vector space $V$. A 
\textit{formal deformation} of $A$ is a pair $(\defo X,\Phi)$ where $\defo X$ is a 
complete and torsion-free $\fser$-algebra and $\Phi:\defo X\otimes_{\fser}\KK\to A$ is an 
isomorphism of $\KK$-algebras. 
Equivalently, a formal deformation of $A$ is a $\fser$-algebra
$\defo X$ with underlying space $V\llbracket t\rrbracket$, such that there exists a 
family 
of maps $F_i\in\Hom_\KK(V\otimes_\KK V,V)$ for all $i\geq1$ satisfying, for all $v,w\in 
V$,
\[
 v\cdot_{\defo X}w = v\cdot_Aw + F_1(v,w)t + F_2(v,w)t^2+\cdots,
\]
where $v\cdot_{\defo X} w$ and $v\cdot_A w$ denote respectively de products of $v$ and 
$w$ in $\defo X$ and in $A$. 
Two formal deformations $(\defo X,\Phi)$ and $(\defo 
X',\Phi')$ of $A$ are \textit{equivalent} if there is an isomorphism $f:\defo X\to\defo 
X'$ of $\fser$-algebras such that the diagram
\begin{align*}
 \xymatrix{
 \defo X\otimes_{\fser}\KK \ar[rr]^-{f\otimes\id} \ar[dr]^-{\Phi}
 & & \defo X'\otimes_{\fser}\KK \ar[ld]_-{\Phi'} \\
 & A &}
\end{align*}
commutes. A formal deformation is \textit{trivial} if it 
is equivalent to $(A\llbracket t\rrbracket,\mathsf{ev}_0)$, where $A\llbracket 
t\rrbracket$ is the algebra of formal series with coefficients in $A$, and 
$\mathsf{ev}_0$ is the map induced by the evaluation at $0$ map.

Suppose $A$ is unital. The Hochshild cohomology of $A$ with values in itself, denoted 
by $\mathrm{H}^\bullet(A,A)$, is the cohomology of the complex 
$\bigoplus_{i\geq0}\Hom_\KK(A^{\otimes i},A)$ with 
differential $d$ given by the formula
\begin{align*}
 d(f)(a_0\otimes \cdots\otimes a_i) &= a_0f(a_1\otimes\cdots\otimes a_{i}) + 
\sum_{j=0}^{i-1}(-1)^jf(a_0\otimes\cdots\otimes a_{j}a_{j+1}\otimes\cdots\otimes 
a_i)\\
&+(-1)^if(a_0\otimes\cdots\otimes a_{i-1})a_i,
\end{align*}
for all $f\in\Hom_\KK(A^{\otimes{i}},A)$ and all $i\geq0$.
The second Hochschild cohomology space $\mathrm{H}^2(A,A)$ describes the 
formal deformations up to first order. More precisely, for every 
nontrivial deformation $\defo X$ of $A$ there exist $n\geq1$ and a formal deformation 
$\defo X'$ of $A$ such that $\defo X$ is equivalent to $\defo X'$ and the family 
of maps $\{F'_i\}_{i\geq1}	$ of $\defo X'$ satisfy 
\begin{itemize}
 \item $F'_i=0$ for all $i\leq n-1$, and
 \item the map $F'_n$ is a $2$-cocycle and its class $[F'_n]$ in $\in\mathrm{H}^2(A,A)$ 
is not zero.
\end{itemize}
See \cite{Ge64}*{Section 3, Proposition 1}. Moreover, if $[G]=[F'_n]$ in 
$\mathrm{H}^2(A,A)$ for some map $G$, then there exists a 
formal deformation $\defo X''$ equivalent to $\defo X$ and whose family of maps $F''_i$ 
satisfy $F''_i=0$ for all $i\leq n-1$ and 
$F''_n = G$.
In particular, if $\mathrm{H}^2(A,A)$ is zero then every formal deformation of $A$ is 
trivial.

If $A$ is graded, a \textit{graded formal deformation} of $A$ is a formal deformation 
$(\defo X,\Psi)$ where $\defo X$ is a graded $\fser$-algebra and $\Psi$ is an 
isomorphism of degree zero. The grading of $A$ induces a grading on its Hochschild 
cohomology, and the space that controls the graded formal deformations is the 
homogeneous component of degree zero of $\mathrm{H}^2(A,A)$.

\subsection*{Lie algebras} The definition of formal deformation for Lie algebras is 
similar to the associative algebra case. Let $\mathfrak g$ be a 
Lie algebra over $\KK$ with underlying vector space $V$. A formal 
deformation of $\mathfrak g$ is a pair $(\defo X,\Psi)$, where $\defo X$ is a complete 
and torsion-free $\fser$-Lie algebra and $\Psi:\defo X\otimes_\fser\KK\to \mathfrak g$ is 
an isomorphism of Lie algebras. Equivalently, it
is a $\KK\llbracket t\rrbracket$-Lie algebra structure 
$\defo X$ on $V\llbracket t\rrbracket$ such that there exists a family of maps 
$F_i\in\Hom_\KK(V\wedge V,V)$ for all $i\geq1$ satisfying
\[
 [v,w]_{\defo X} = [v,w]_{\mathfrak g} + F_1(v,w)t + F_2(v,w)t^2+\cdots
\]
for all $v,w\in V$. In this setting, the space that controls  the formal deformations of 
a Lie algebra $\mathfrak g$ is the cohomology space $\mathrm{H}^2_{Lie}(\mathfrak 
g,\mathfrak g)$, which is the cohomology of the complex 
$\bigoplus_{i\geq0}\Hom_\KK(\Lambda^i\mathfrak g,\mathfrak g)$ where the differential is
\begin{align*}
 d(f)(v_0\wedge\cdots\wedge v_{i})&= 
\sum_{j=0}^{i}(-1)^{j}[v_j,f(v_1\wedge\cdots \wedge\hat{v}_j\wedge\cdots\wedge 
v_{i})] \\&\hskip0.2cm+ 
\sum_{j<k}(-1)^{j+k}f([v_j,v_k]\wedge\cdots\wedge\hat{v}_j\wedge\cdots\wedge\hat{v}_k 
\wedge\cdots\wedge v_{i}),
\end{align*}
for all $f\in\Hom_\KK(\Lambda^{i}\mathfrak g,\mathfrak g)$ and all $i\geq0$.
If $\defo X$ is a nontrivial formal deformation of $\mathfrak g$, then 
it is equivalent to a formal deformation $\defo X'$ with family of maps $F'_i$ satisfying 
$F'_i=0$ for all $i\leq n-1$ and $[F'_n]\neq 0$ in $\mathrm{H}^2_{Lie}(\mathfrak 
g,\mathfrak g)$, for some $n\geq1$. As in the case of associative algebras, if 
$[G]=[F'_n]$, there exists a formal deformation $\defo X''$ equivalent to $\defo X$ and 
whose family of maps $F''_i$ satisfy $F''_i=0$ for all $i\leq n-1$ and $F''_n=G$.

\section{Degenerations of finite dimensional associative algebras}
\label{sec:degAss}

Let $V$ be a finite dimensional vector space over $\KK$, and let $\ass[V]{\KK}$ be 
the variety of associative algebra structures on $V$. More precisely, 
$\ass[V]{\KK}$ is the subset of $\Hom_\KK(V\otimes V, V)$ consisting of the elements 
$\varphi$ such that 
\[\varphi\circ(\varphi\otimes\id) = \varphi\circ(\id\otimes\varphi)\]
as morphisms from $V\otimes V\otimes V$ to $V$. The algebraic group 
of linear isomorphisms $\GL(V)$ acts on $\ass[V]{\KK}$ via de formula $g\cdot \varphi 
= g\circ\varphi\circ(g^{-1}\otimes g^{-1})$ for $g\in\GL(V)$ and 
$\varphi\in\ass[V]{\KK}$. The orbits are in natural 
bijection with the isomorphism classes of associative algebras of dimension $\dim_\KK V$.
For such an algebra $A$ denote by $\orb A$ its corresponding 
orbit in $\ass[V]{\KK}$ and by $\clorb A$ its Zariski closure. Let $A$ and $B$ be 
$\KK$-algebras in $\ass[V]{\KK}$.  The algebra $B$ is said to be a \textit{degeneration} 
of $A$ if $\orb B\subseteq\clorb A$.

\begin{theorem}
\label{teo:caso_as}Let $A$ and $B$ be associative $\KK$-algebras of the same dimension 
regarded as points of $\ass[V]{\KK}$. 
The algebra $B$ is a degeneration of $A$ if and only if there 
exists a Gerstenhaber formal deformation $\defo X$ of $B$ such that 
$\extesc{\defo X}{\fser}$ is isomorphic to $\extesc{A}{\KK}$ as algebras over 
$\KK(\!(t)\!)$. 

\end{theorem}
\begin{proof}
 Let $(X,G)=(\ass[V]{\KK},\GL(V))$. Let $\func$ be the functor of Theorem \ref{teo:main}. 
Note that $\func(\fser) = 
\Hom(\mathsf{Spec}(\fser),X)$ is the set of associative $\fser$-algebra structures on 
$V\otimes_\KK \fser$. Since $V$ is finite dimensional, the space $V\otimes_\KK \fser$ is 
equal to  $V\llbracket t\rrbracket$ . The result follows easily from Theorem 
\ref{teo:main} and the definition of Gerstenhaber 
formal deformation.
\end{proof}

There is an immediate consequence of Theorem \ref{teo:caso_as}. Let $A$ and $B$ be finite 
dimensional algebras. Write $B \leq A$ if there exists a formal deformation $\defo X$ of 
$B$ such that $\defo X\otimes_{\fser}\lser\cong A\otimes_\KK\lser$. By Theorem 
\ref{teo:caso_gras}, $B \leq A$ 
if and only if $B$ is a degeneration of $A$. The binary relation defined by the 
geometric degeneration relation is a partial order on the set of orbits. As a consequence 
we obtain the following result.

\begin{proposition}
 The binary relation $\leq$ is a partial order on the isomorphism classes of finite 
dimensional associative algebras.
\end{proposition}

We do not know whether there is a purely algebraic proof of this fact nor if there is a 
larger class of algebras than finite dimensional algebras where the relation $\leq$ 
extends and for which a similar result holds.

\medskip

For a finite dimensional graded vector space $V=\bigoplus_{i\in\ZZ}V_i$ over $\KK$, 
define $\gass[V]{\KK}$ to be the variety of graded algebra structures on $V$. The 
algebraic group $\gGL(V)$ of linear isomorphisms of degree $0$ acts on 
$\gass[V]{\KK}$ and its orbits are in correspondence with the isomorphism classes of 
finite dimensional graded algebras with Hilbert series $\sum_{i\in\ZZ}\dim_\KK(
V_i)\,t^i$. For two such algebras $A$ and $B$, we say that \textit{ $A$ is a graded 
degeneration of $B$} if $\orb A \subseteq \clorb B$ in $\gass[V]{\KK}$. The proof of 
Theorem \ref{teo:caso_as} can be easily adapted to prove the following theorem.

\begin{theorem}
\label{teo:caso_gras}Let 
$A$ and $B$ be graded associative $\KK$-algebras regarded as points of $\gass[V]{\KK}$. 
The algebra $B$ is a graded degeneration of $A$ if and only if there 
exists a graded Gerstenhaber formal deformation $\defo X$ of $B$ such that 
$\extesc{\defo 
X}{\fser}$ is isomorphic to $\extesc{A}{\KK}$ as graded algebras over $\KK(\!(t)\!)$. 
\end{theorem}

\medskip

Next we turn to $N$-Koszul algebras.
Let $A=\bigoplus_i A_i$ be a finite dimensional graded 
$\KK$-algebra with unit, generated as an algebra by $A_1$. In particular $A_0=\KK$ and 
$J=\bigoplus_{i\geq1}A_i$ is a maximal ideal. This induces an $A$-module structure on 
$\KK$ by identifying it with $A/J$.

Let $N\geq2$ and define the function $n:\NN\to\NN$ by $n(2i)=iN$ and $n(2i+1)=iN+1$. 
The grading of $A$ induces a grading on $\Tor^A_i(\KK,\KK)$ for all $i$, whose 
homogeneous component of degree $j$ we denote by $\Tor^A_{ij}(\KK,\KK)$.
The algebra $A$ is \textit{$N$-Koszul} if $\Tor^A_{ij}(\KK,\KK)$ is zero for 
all $i,j\geq0$ such that $j\neq n(i)$. These algebras were 
introduced by Berger in \cite{Be01}, generalising the classical case due to Priddy which 
corresponds to $N=2$ \cite{Pr70}. See also \cite{GMMVZ04} for a further generalisation to 
the non local case.

A result of Drinfeld \cite{Dr92} (see also \cite{PP05}*{Chapter 6}) has the following 
corollary.

\begin{corollary}
\label{coro:drinfeld}
Let $A$ and $B$ be finite dimensional graded $\KK$-algebras with unit such that $B$ is a 
graded degeneration of $A$. If $B$ is $2$-Koszul then $A$ is $2$-Koszul.
\end{corollary}
In other words, a finite dimensional algebra $A$ is $2$-Koszul if at least one of the 
algebras contained in $\clorb A$ is $2$-Koszul. This result, besides being very 
interesting in its own, is the key to the proof that the cohomology algebra of 
the complement of a supersolvable complex hyperplane arrangement is $2$-Koszul, which is 
an important theorem in algebraic combinatorics \cite{SY97}.

We extend Corollary \ref{coro:drinfeld} to the setting of $N$-Koszul algebras.
\begin{proposition}
\label{prop:koszul}
 Let $A$ and $B$ be finite dimensional graded algebras with unit such that $B$ is a 
graded degeneration of $A$. If $B$ is $N$-Koszul, then $A$ is $N$-Koszul.
\end{proposition}
For its proof we need the following Lemma.
\begin{lemma}
\label{lemma:complejos}
 Let $P$ be a complex of finitely generated free $\fser$-modules. For any $i\in\ZZ$, if 
$H_i(P/tP)$ vanishes, then so does $H_i(P)$.
\end{lemma}
\begin{proof}
By the canonical isomorphism $P\otimes_{\fser}\KK\cong P/tP$ and the K\"unneth formula 
there is an injection $H_i(P)/tH_i(P)\to H_i(P/tP)=0$. Then 
$H_i(P)$ vanishes by the Nakayama Lemma.
\end{proof}

\begin{proof}[Proof of Proposition \ref{prop:koszul}.]
By Theorem \ref{teo:caso_gras} there exists a graded Gerstenhaber formal deformation 
$\defo X$ of $B$ such that $\defo X\otimes_{\fser}\lser$ is isomorphic to 
$A\otimes_\KK\lser$ as graded $\lser$-algebras. Since $B$ is generated in degree $1$, the 
Nakayama Lemma and the following isomorphisms of graded algebras
\[
 B \cong \defo X\otimes_{\fser}\KK \cong \defo X/ t\defo X,
\]
imply that $\defo X$ is generated in degree $1$ as a $\fser$-algebra. As a consequence of 
the isomorphism 
$A\otimes_\KK\lser\cong\defo X\otimes_{\fser}\lser$, the algebra $A$ is also 
generated in degree $1$.

Since $\defo X$ is free of finite rank as a $\fser$-module, there exists a resolution 
$P$ of $\fser$ by finitely 
generated projective graded left $\defo X$-modules. So $\Tor^{\defo X}_i(\fser,\fser)$ is 
isomorphic to $\mathrm{H}_i(\fser\otimes_{\defo X}P)$. On the other hand, by the Künneth 
formula, the complex $P\otimes_{\fser}\KK$ is a resolution of $\KK$ by finitely generated 
graded projective 
$B$-modules and therefore $\Tor^B_i(\KK,\KK)$ is isomorphic to 
$\mathrm{H}_i(\KK\otimes_{B}P\otimes_{\fser}\KK)$. Note that
\begin{align*}
 \KK\otimes_B P \otimes_{\fser}\KK &\, \cong \, \KK\otimes_{\defo X} P\otimes_{\fser}\KK 
\\&\,\cong\, \KK\otimes_{\defo X} P \\&\,\cong\, \KK\otimes_{\fser}\fser\otimes_{\defo 
X}P \\&\,\cong\, \fser\otimes_{\defo X}P/t\left(\fser\otimes_{\defo X}P\right).
\end{align*}
Let $i$ and $j$ be natural numbers such that $j\neq n(i)$. Since $B$ is Koszul, the space 
$\Tor^B_{ij}(\KK,\KK)$ vanishes. So, applying Lemma \ref{lemma:complejos} to the 
homogeneous component of degree $j$ of the complex $\fser\otimes_{\defo X} P$, we 
obtain that $\Tor^{\defo X}_{ij}(\fser,\fser)=0$.
On the other hand,
\begin{align*}
 \Tor^A_i(\KK,\KK)\otimes_\KK\lser &\cong \Tor^{A\otimes_\KK\lser}_i(\lser,\lser)\\
 &\cong \Tor^{\defo X\otimes_{\fser}\lser}_i(\lser,\lser)\\
 &\cong \Tor^{\defo X}_i(\fser,\fser)\otimes_{\fser}\lser.
\end{align*}
As a consequence $\Tor^A_{ij}(\KK,\KK)\otimes_\KK\lser$ vanishes, implying that $A$ is 
$N$-Koszul.
\end{proof}

\section{Degenerations of finite dimensional Lie algebras}

Let $V$ be a finite dimensional vector space over an algebraically closed field $\KK$. 
Let $\liea[V]{\KK}$ be the variety of Lie algebra structures on $V$, that is, the 
set of maps $\varphi\in\Hom_\KK(\Lambda^2 V,V)$ satisfying the Jacobi identity
\[
 \varphi(u\wedge\varphi(v\wedge w)) + \varphi(v\wedge\varphi(w\wedge u))
 + \varphi(w\wedge\varphi(u\wedge v) = 0,
\]
for all $u,v,w\in V$. For $g\in\GL(V)$ and $\varphi\in\liea[V]{\KK}$, the map 
$g\circ\varphi\circ(g^{-1}\wedge g^{-1})$ belongs again to $\liea[V]{\KK}$. The orbits of 
this action are in natural correspondence with the isomorphism classes of Lie algebras of 
dimension $\dim_\KK V$. The next result is obtained by applying Theorem \ref{teo:main} to 
the pair $(\liea[V]{\KK},\GL(V))$. The proof is analogous to the case of associative 
algebras and we omit the details. For a Lie algebra $\mathfrak g$, we denote by 
$\clorb{\mathfrak g}$ the Zariski closure in $\liea[V]{\KK}$ of the orbit $\orb{\mathfrak 
g}$. A Lie algebra \textit{$\mathfrak h$ is a degeneration of $\mathfrak g$} if 
$\orb{\mathfrak h}\subseteq \clorb{\mathfrak g}$.

\begin{theorem}
\label{teo:casoLie}
 Let $\mathfrak g$ and $\mathfrak h$ be Lie algebras over $\KK$ of the same dimension. 
The algebra $\mathfrak h$ is a degeneration of $\mathfrak g$ 
if and only if there exists a Gerstenhaber formal deformation $\defo X$ of 
$\mathfrak h$ such that $\extesc{\defo X}{\fser}$ is isomorphic to $\extesc{\mathfrak 
g}{\KK}$ as Lie algebras over $\KK(\!(t)\!)$. 
\end{theorem}

We give an application of Theorem \ref*{teo:casoLie} to the description of 
the degenerations of Lie algebras of dimension $3$. These have been already 
computed by Agaoka in \cite{Aga99}. We show how Theorem \ref*{teo:casoLie} can be used to 
obtain Agaoka's result by means of homological algebra.

\subsection*{Classification of \texorpdfstring{$3$}{TEXT}-dimensional Lie algebras}
We begin by recalling the classification of the isomorphism classes of $3$-dimensional 
Lie algebras. See for example 
\cite{FH91}.

If $\mathfrak g$ is a Lie algebra of dimension $3$ over an algebraically 
closed field $\KK$, then $\mathfrak g$ is isomorphic to one of the algebras in 
the following list.

\begin{center}
\captionof{table}{}
\label{tabla:1}
\begin{tabular}{r|l}
  & non-trivial bracket operations \\ \hline
  $L_0$ & \\
  $L_1$ & $[x,y]=z$\\
  $L_2$  & $[x,y]=y$\\
  $L_3$  & $[x,y]=y$, $[x,z]=y+z$\\
  ($\alpha\in\KK^\times$) $L_4(\alpha)$ & $[x,y]=y$, $[x,z]=\alpha z$\\
  $L_5$ & $[x,y]=y$, $[x,z]=-z$, $[y,z]=x$
\end{tabular}
\end{center}
For $0\leq i,j\leq 5$, the algebras $L_i$ and $L_j$ are not 
isomorphic if 
$(i,j)\neq(4,4)$, and $L_4(\alpha)$ is isomorphic to $L_4(\beta)$ if and only 
if $\alpha=\beta^{-1}$ or $\alpha=\beta$.

Note that if $\FF$ is an 
algebraically closed field such that $\KK\subseteq \FF$, for example the 
algebraic closure of $\KK(\!(t)\!)$, then every Lie algebra over 
$\FF$ is either isomorphic to $L_i\otimes_\KK\FF$, for some $i\neq4$, or it is 
isomorphic to $L_4(\alpha)$ for some $\alpha\in\FF$.

Let $\mathfrak g$ be a $3$-dimensional Lie algebra. There is a method given in 
\cite{FH91} to determine which algebra of the above list is 
isomorphic to $\mathfrak g$, which we shall now describe. Let $n$ be the rank of 
$\mathfrak g$, that is, the dimension of $[\mathfrak g,\mathfrak g]$. Note that regarding 
the bracket as a map $[\, , \,]:\Lambda^2\mathfrak g\to\mathfrak g$, we obtain that 
$n=\dim_\KK(\im[\, ,\,])$. 
If $n=0$, then $\mathfrak g$ is isomorphic to $L_0$, and if $n=3$, then $\mathfrak 
g$ is isomorphic to $L_5$.

Suppose now $n=1$. In this case there exists $x\in\mathfrak g$ such that $\{x\wedge 
y: y\in\mathfrak g\} = \ker[\, ,\,]$. Let $y,z\in\mathfrak g$ be such 
that $\{x,y,z\}$ is a basis of $\mathfrak g$ and let $a,b,c\in\KK$ be such 
that $[y,z]=ax + by + cz$. If $b=c=0$, then $\mathfrak g$ is isomorphic to 
$L_1$, and if $(b,c)\neq(0,0)$, then $\mathfrak g$ is isomorphic to $L_2$.

Suppose $n=2$. 
Let $\{x,y,z\}$ be a basis of $\mathfrak 
g$ such that $\{y,z\}$ is a basis of $[\mathfrak g,\mathfrak g]$. It is not difficult to 
see that $\mathsf{ad}(x):[\mathfrak g,\mathfrak g]\to[\mathfrak g,\mathfrak g]$ is an 
isomorphism. If 
$\mathsf{ad}(x)$ is not diagonalizable, then $\mathfrak g$ is isomorphic to 
$L_3$. If $\mathsf{ad}(x)$ is diagonalizable and $\lambda_1,\lambda_2$ are its 
eigenvalues, then $\mathfrak g$ is isomorphic to $L_4(\lambda_1/\lambda_2)$.

\subsection*{Cohomology of \texorpdfstring{$3$}{TEXT}-dimensional Lie algebras}
Let $\mathfrak g$ be a Lie algebra of dimension $3$ over an 
algebraically closed field $\KK$. Recall that the spaces 
$\mathrm{H}^i_{\mathsf{Lie}}(\mathfrak g,\mathfrak g)$ are the homology spaces 
of the complex
\[
 \xymatrix{
  0 \ar[r]&
  \mathfrak g \ar[r]^-{d_1} & 
  \Hom_\KK(\mathfrak g,\mathfrak g) \ar[r]^-{d_2} &
  \Hom_\KK(\Lambda^2\mathfrak g,\mathfrak g) \ar[r]^-{d_3} &
  \Hom_\KK(\Lambda^3\mathfrak g,\mathfrak g) \ar[r] & 
  0,
 }
\]
We may identify the spaces $\Hom_\KK(\Lambda^i\mathfrak g,\mathfrak g)$ and 
$\Lambda^i\mathfrak g^*\otimes_\KK\mathfrak g$ for all $i$, where $\mathfrak g^* = 
\Hom_\KK(\mathfrak g, \KK)$. Let $\{x,y,z\}$ be a 
basis of $\mathfrak g$ and $\{\hat x,\hat y,\hat z\}\subseteq \mathfrak g^*$ its dual 
basis. The following table shows for each algebra $\mathfrak g$ a list of 
elements in $\Lambda^2\mathfrak g^*\otimes_\KK \mathfrak g$ such that their classes in 
$\mathrm{H}^2_{\mathsf{Lie}}(\mathfrak g,\mathfrak g)$ form a basis.

\medskip

\begin{center}
\captionof{table}{}
\label{tabla:2}
 \begin{tabular}{r|l}
  $L_1$ & $\hat y\wedge \hat x\otimes x$, $\hat y\wedge\hat x \otimes y$, $\hat 
z\wedge\hat x\otimes y$, $\hat z\wedge\hat y\otimes x$, $\hat z\wedge\hat 
x\otimes x - \hat z\wedge\hat y\otimes y$ \\
  $L_2$ & $\hat z\wedge\hat x\otimes z$ \\
  $L_3$ & $\hat y\wedge\hat x\otimes z$ \\
  $(\alpha\neq -1,1)$ $L_4(\alpha)$ & $\hat z\wedge \hat x\otimes z$ \\
  $L_4(-1)$ & $\hat z\wedge\hat x\otimes z$, $\hat z\wedge\hat y\otimes x$\\
  $L_4(1)$ & $\hat y\wedge\hat x\otimes z$, $\hat z\wedge \hat x\otimes y$, $\hat 
z\wedge\hat x\otimes z$
 \end{tabular}
\end{center}
The space $\mathrm{H}^2_{\mathsf{Lie}}(L_0,L_0)$ is $\Lambda^2L_0^*\otimes_\KK L_0$ 
and the space $\mathrm{H}^2_{\mathsf{Lie}}(L_5,L_5)$ is zero. This is 
obtained by straightforward computations using the above complex.

\subsection*{Degenerations of \texorpdfstring{$3$}{TEXT}-dimensional Lie algebras}

Let $\KK$ be an algebraically closed field. From now on we assume all 
$3$-dimensional Lie algebras over $\KK$ to have the same underlying vector space $V$ with 
basis $\{x,y,z\}$. The next proposition shows in a particular case how Theorem 
\ref{teo:casoLie} and the theory of Gerstenhaber formal deformations can be used to 
determine the Zariski closure of an orbit.

\begin{proposition}
\label{prop:lie_l2}
 If $\orb{L_2}$ is contained in $\clorb{\mathfrak g}$ for some Lie algebra $\mathfrak g$, 
then $\orb g= \orb{L_2}$.
\end{proposition}

\begin{proof}

Let $\mathfrak h=L_2$. Suppose that $\orb{\mathfrak h}\subseteq\clorb{\mathfrak g}$ for 
some Lie algebra $\mathfrak g$ not isomorphic to $\mathfrak h$. By Theorem 
\ref{teo:casoLie} there exists a nontrivial Gerstenhaber formal deformation $\defo X$ of 
$\mathfrak h$ such that 
\[\extesc{\defo X}{\fser}\cong \extesc{\mathfrak g}{\KK}.\]  
By Table \ref*{tabla:2} and the remark at the end of Section \ref{sec:gerst} we may 
assume without loss of generality that there exist 
$\lambda\in\KK^\times$ and $n\geq1$ such that for all $v,w\in V$,
\begin{align}
\label{ecuacion1}
 [v,w]_{\defo X} = [v,w]_{\mathfrak h} + F_n(v,w)t^n + F_{n+1}(v,w)t^{n+1}+\cdots,
\end{align}
for some maps $F_i:V\wedge V\to V$ for all $i\geq n$ such that $F_n=\lambda \hat 
z\wedge\hat x\otimes z$.

Since $V$ is finite dimensional, the spaces $V\llbracket t 
\rrbracket$ and $V\otimes_\KK\fser$ are isomorphic and so $V\llbracket t 
\rrbracket$ is the free $\fser$-module with basis $\{x,y,z\}$. Specializing Equation 
\eqref{ecuacion1} at $v,w\in\{x,y,z\}$ we obtain that there exist elements 
$f_i,g_i,h_i\in\KK\llbracket t\rrbracket$, for $i\in\{1,2,3\}$, such that
\begin{align*}
 &[y,x]_\defo X = t^{n+1}f_1x + (-1+t^{n+1}f_2)y + t^{n+1}f_3z,\\
 &[z,x]_\defo X = t^{n+1}g_1x + t^{n+1}g_2y + (\lambda t^n + t^{n+1}g_3)z,\\
 &[z,y]_\defo X = t^{n+1}h_1x + t^{n+1}h_2y + t^{n+1}h_3z.
\end{align*} 
Since the minor
\[(-1+t^{n+1}f_2)(\lambda t^n + g_3t^{n+1}) - t^{n+1}g_2t^{n+1}f_3 \equiv -\lambda t^n 
\,\textrm{ mod }t^{n+1},\] we deduce that it is different from zero and therefore the set 
$\{x,[y,x]_\defo X,[z,x]_\defo X\}$ is
basis of $\extesc{\defo X}{\fser}$. As a 
consequence the rank of $\extesc{\defo X}{\fser}$ is at least 
$2$. Thus $\mathfrak g$ is of rank at least $2$ and not isomorphic to $L_0, L_1$ or 
$L_2$. 

Suppose the rank of $\mathfrak g$ is $2$, that is, $\mathfrak g$ is isomorphic 
either to $L_3$ or to $L_4(\alpha)$ for some $\alpha\in\KK^\times$. Let $\FF$ be the 
algebraic closure of $\KK(\!(t)\!)$. Then, the algebra $\defo X\otimes_{\fser}\FF$ is 
isomorphic either to 
$L_3\otimes_\KK\FF$ or to 
$L_4(\alpha)\otimes_{\KK(\!(t)\!)}\FF$ for some $\alpha\in\KK^\times$. 
In any case, if $\mu_1,\mu_2\in\FF$ are the 
eigenvalues of $\mathsf{ad}(x)$, we obtain $\mu_1=\beta\mu_2$ for some 
$\beta\in\KK^\times$.
The matrix of the map 
$\mathsf{ad}(x)$ in the basis $\{[y,x],[z,x]\}$ is

\[
 \begin{pmatrix}
  1-t^{n+1}f_2 & -t^{n+1}g_2 \\
  -t^{n+1}f_3 & -\lambda t^n - t^{n+1}g_3
 \end{pmatrix}
\]
and its trace and determinant are 
\begin{align*}
&(1+\beta)\mu_2 = -\mathrm{tr}(\mathsf{ad}(x))=-1+t^{n+1}f_2 +\lambda t^n + 
t^{n+1}g_3,\\
&\beta\mu_2^2 = \det(\mathsf{ad}(x))= (1-t^{n+1}f_2)(-\lambda t^n - t^{n+1}g_3).
\end{align*}
Recall that we denote $v:\KK(\!(t)\!)\to\ZZ$ the valuation of the Laurent series field. 
The first equality implies that $\mu_2\in\KK\llbracket t\rrbracket$ and that 
$v(\mu_2)=0$. On the other hand, the second equality implies that $v(\mu_2)>1$. This is a 
contradiction and we deduce $\mathfrak g$ is not of rank $2$.

Let us see that $\mathfrak g$ does not have rank $3$. From the equation
\[
 0=[x,[y,z]_\defo X]_\defo X + [y,[z,x]_\defo X]_\defo X + 
[z,[x,y]_\defo X]_\defo X,
\]
we deduce 
\begin{align*}
 &h_1 - \lambda h_1 t^n + t^{n+1}(f_1h_2 - f_2h_1 + g_1h_3 - g_3h_1)=0,\\
 &g_1 + \lambda h_2 t^n - t^{n+1}(f_2g_1 - f_1g_2 + g_2h_3 - g_3h_2)=0,\\
 &h_3 - \lambda f_1 t^n + t^{n+1}(f_3g_1 - f_1g_3 + f_3h_2 - f_2h_3)=0.
\end{align*}
From these equations we can express $h_1$, $g_1$ and $h_3$ as polynomials on 
$f_1,f_2,f_3,g_2,g_3,h_1$ and $h_2$. Using these identities, the determinant of the matrix
\[
 \begin{pmatrix}
  t^{n+1}f_1 & -1+t^{n+1}f_2 & t^{n+1}f_3 \\
  t^{n+1}g_1 & t^{n+1}g_2 & \lambda t^n + t^{n+1}g_3\\
  t^{n+1}h_1 & t^{n+1}h_2 & t^{n+1}h_3
 \end{pmatrix}
\]
is zero. We performed these computations using \textit{Mathematica 10} \cite{WR14}.
Thus $\extesc{\defo X}{\fser}$ is not of rank $3$ and therefore 
$\mathfrak g$ is not of rank $3$.

As a consequence, the rank of the algebra $\mathfrak g$ is different 
from $0,1,2$ or $3$. This is a contradiction coming from the assumption that 
$\mathfrak g$ is not isomorphic to $\mathfrak h$.
\end{proof}

The methods of Proposition \ref{prop:lie_l2} give an homological algebra tool to decide 
whether the degeneration relation $\orb{\mathfrak h} \subseteq \clorb{\mathfrak g}$ 
holds for some Lie algebras $\mathfrak h$ and $\mathfrak g$. Since each set 
$\clorb{\mathfrak g}$ is a union of orbits, a complete description of the degeneration 
relation determines them. Using this technique we were able to recover the full 
description of the sets $\clorb{\mathfrak g}$ for all $3$-dimensional Lie algebras 
given by Agaoka in \cite{Aga99}*{Proposition 5} with geometric methods. We omit the 
details which involve computations similar to the ones given in the proof of Proposition 
\ref{prop:lie_l2}.

\end{document}